\documentclass[12pt,a4paper]{article}
\usepackage{a4wide,amsfonts,amsmath,latexsym,amssymb,euscript,eufrak,graphicx,units,mathrsfs}
\usepackage[utf8]{inputenc}
\usepackage{amsmath}
\usepackage{amsfonts}
\usepackage{amssymb}
\usepackage{amsthm}

\renewcommand{\theequation}{\arabic{section}.\arabic{equation}}

\usepackage[usenames,dvipsnames]{color}

\newcommand{\R}{\mathbb{R}}

\newtheorem{theorem}{Theorem}[section]

\newtheorem{lemma}{Lemma}[section]
\newtheorem{remark}{Remark}[section]
\newtheorem{proposition}{Proposition}[section]

\begin{document}

\begin{center}

\Large{Least squares estimator for non-ergodic
Ornstein-Uhlenbeck processes driven by  Gaussian processes }

\bigskip

\normalsize{Mohamed El Machkouri\footnote{ Laboratoire de
Math\'ematiques Rapha\"{e}l Salem UMR CNRS 6085, Universit\'e de
Rouen, France. Email: mohamed.elmachkouri@univ-rouen.fr}, Khalifa
Es-Sebaiy\footnote{National School of Applied Sciences-Marrakesh,
Cadi Ayyad University Av. Abdelkrim Khattabi, 40000,
Gu\'eliz-Marrakech, Morocco. Email: k.essebaiy@uca.ma} and Youssef
Ouknine\footnote{National School of Applied Sciences-Marrakesh, Cadi
Ayyad University Av. Abdelkrim Khattabi, 40000, Gu\'eliz-Marrakech,
Morocco. Email: ouknine@uca.ma}}
\end{center}

\

{\small \noindent {\bf Abstract}: The statistical analysis for
equations driven by fractional Gaussian process (fGp)  is relatively
recent. The development of stochastic calculus with respect to the
fGp allowed to study such models. In the present paper we consider
the drift parameter estimation problem for the non-ergodic
Ornstein-Uhlenbeck process defined as $dX_t=\theta X_tdt+dG_t,\
t\geq0$  with an unknown parameter  $\theta>0$, where $G$ is  a
Gaussian process. We provide sufficient  conditions, based on the
properties of $G$, ensuring
 the strong consistency and the asymptotic distribution of our estimator $\widetilde{\theta}_t$ of
 $\theta$ based on the
observation $\{X_s,\ s\in[0,t]\}$ as $t\rightarrow\infty$.
 Our approach offers an elementary, unifying proof of \cite{BEO},
  and it allows to extend the result of \cite{BEO} to the case  when $G$ is a fractional
   Brownian motion with Hurst parameter $H\in(0,1)$. We also discuss the cases of subfractional
   Brownian motion and bifractional Brownian motion.\\


\section{Introduction}
While the statistical inference of Ito's type diffusions has a long
history, the statistical analysis for equations driven by fractional
Gaussian process is relatively recent. The development of stochastic
calculus with respect to the fGp has allowed to study such models.
We will recall  several approaches to estimate the parameters in
fractional models but we mention that the list below is not
exhaustive:
\begin{description}
\item{- } The MLE approach in \cite{KL}, \cite{TV}: In general the techniques used to construct
maximum likelihood estimators (MLE)  for the drift parameter
 are based on Girsanov's transforms for fBm and depend on  the properties of the deterministic
fractional operators (determined by the Hurst parameter) related to
the fBm. In this case, the MLE is not easily computable.
\item{- } A least squares approach has been proposed in \cite{HN}: The study of the asymptotic properties of
 the estimator is based on certain criteria formulated in terms of the Malliavin calculus. In the ergodic case,
  the statistical inference for several fractional Ornstein-Uhlenbeck (fOU) models has been recently developed
   in the papers \cite{HN}, \cite{AM}, \cite{AV}, \cite{HS}, \cite{EEV}, \cite{CE}. The case of non-ergodic fOU process of the first kind and of the second kind  can be found in \cite{BEO} and  \cite{EET} respectively.
\item{- } Method of moments:  A new idea has been provided in \cite{EV}, to develop
the statistical inference for stochastic differential equations related to stationary Gaussian processes
 by proposing a suitable criteria. This approach is based on the Malliavin calculus, and it makes in principle
 the estimators easier to be simulated. Moreover, as an application, the models discussed
 in \cite{HN}, \cite{AM}, \cite{AV}, \cite{EEV} have been studied in \cite{EV} by using this approach.
\end{description}
In this paper, we consider the non-ergodic  Ornstein-Uhlenbeck
process $X=\left\{X_t, t\geq0\right\}$ given by the following linear
stochastic differential equation
\begin{eqnarray}\label{OU}X_0=0;\quad  dX_t=\theta X_tdt+dG_t,\quad t\geq0,
\end{eqnarray}where  $G$ is a Gaussian process and  $\theta>0$   is an unknown parameter. \\
A  problem here is to estimate the parameter $\theta$ when one
observes the whole trajectory of $X$. In the case when the process
$X$  has H\"older continuous paths of order $\delta\in(\frac12,1]$
we can consider the following least squares estimator (LSE)
\begin{eqnarray}\label{(LSE)} \widehat{\theta}_t =\frac{\int_0^tX_sdX_s}{\int_0^tX_s^2ds},\quad t\geq0,
\end{eqnarray}
as estimator of  $\theta$, where the integral with respect to $X$ is
a Young integral (see Appendix). The estimator  $\widehat{\theta}_t$
is obtained by the least squares technique, that is, $
\widehat{\theta}_t$ (formally) minimizes
\[\theta \longmapsto \int_0^t\left|\dot{{X}}_s-\theta X_s\right|^2ds.\]
Moreover, using  the formula (\ref{Integration by parts}) we can
rewrite  $\widehat{\theta}_t$  as follows,
 \begin{eqnarray}\label{LES2} \widehat{\theta}_t  =\frac{X^2_t}{2\int_0^tX_s^2ds},\quad t\geq0.
\end{eqnarray}
Motivated by   (\ref{LES2}) we propose to use, in the general case,
the right hand of (\ref{LES2}) as a statistic to estimate the drift
coefficient $\theta$ of the equation (\ref{OU}). More precisely, we
define
 \begin{eqnarray}\label{estimator} \widetilde{\theta}_t =\frac{X^2_t}{2\int_0^tX_s^2ds},\quad t\geq0.
\end{eqnarray}
This estimator $\widetilde{\theta}_t$ may exist even if  $X$ does
not have H\"older continuous paths of order $\delta\in(\frac12,1]$.

 We shall provide sufficient   conditions, based on the properties of $G$, under which   the  estimator $\widetilde{\theta}_t$ is consistent   (see   Theorem \ref{consistency}), and the limit distribution of $\widetilde{\theta}_t$ is a standard Cauchy distribution (see Theorem
\ref{convergence distribution}). \\

\noindent {\bf Examples of the Gaussian process G.}

  \emph{\underline{Fractional Brownian motion}}:\\
Suppose that the process $G$  given in (\ref{OU})  is a fractional
Brownian motion with Hurst parameter $H\in(0,1)$. By assuming that
$H>\frac12$,  \cite{BEO}  studied  the    LSE  $\widehat{\theta}_t$
which coincides, in this case, with $\widetilde{\theta}_t$ by Remark
\ref{remark1}. In this paper, we extend the  result of \cite{BEO} to
the case $H\in(0,1)$. Moreover, we offer an elementary  proof (see
Section 3.1).

  \emph{\underline{Sub-fractional Brownian motion}}:\\
Assume that the process $G$  given in (\ref{OU})  is a subfractional
Brownian motion with  parameter $H\in(0,1)$. For $H>\frac12$,  using
an idea of \cite{BEO}, \cite{mendy}  studied the    LSE
$\widehat{\theta}_t$ which also coincides with
$\widetilde{\theta}_t$. But the proof of Lemma 4.3 in \cite{mendy}
relies on a possibly awed technique  because the passage from line
-7 to -6 on page 671  does not allow to obtain the convergence of
$E\left[\left(e^{-\theta t}\int_0^te^{\theta
s}dS^H_s\right)^2\right]$ as $t\rightarrow\infty$. In the present
paper, we give a solution of this problem and we extend the result
to   $H\in(0,1)$ (see Section 3.2).

 \emph{\underline{Bifractional Brownian motion}}:\\
 To the best of our knowledge there is no study of the  problem of estimating the drift of (\ref{OU}) in the
case when $G$ is a bifractional Brownian motion with  parameters
$(H,K)\in (0,1)^2$. Section 3.3 is devoted to  this question.

\section{Asymptotic behavior of the  estimator}
 Let $G = \left(G_t,t\geq0\right)$  be a continuous centered Gaussian process defined on some probability
 space $(\Omega, \mathcal{F}, P)$ (Here,
and throughout the text, we assume that $\mathcal{F}$ is the
sigma-field generated by $G$). The following assumptions are
required.
\begin{itemize}
\item[$(\mathcal{H}1)$] The process $G$   has H\"older continuous paths of order $\delta\in(0,1]$.
\item[$(\mathcal{H}2)$] For every $t\geq0$, $E\left(G_t^2\right)\leq ct^{2\gamma}$ for some positive constants  $c$ and $ \gamma$.
\end{itemize}
\subsection{Strong consistency}
We will prove that the  estimator $\widetilde{\theta}_{t}$ given by
(\ref{estimator}) is strongly consistent.
\\ It is clear that the linear equation (\ref{OU}) has the following
explicit solution
\begin{eqnarray}\label{explicit solution}X_t=e^{\theta t}\int_0^te^{-\theta s}dG_s,\qquad t\geq0,
\end{eqnarray}
where the integral is interpreted in the Young sense (see Appendix).\\
Suppose $(\mathcal{H}1)$ holds. Applying the formula
(\ref{Integration by parts}) we can write
\begin{eqnarray}\label{representation of X} X_t=G_t+\theta e^{\theta
t}Z_t,\qquad t\geq0,
\end{eqnarray}
where \begin{eqnarray}Z_t:=\int_0^te^{-\theta s}G_sds,\qquad
t\geq0.\label{expr Z}\end{eqnarray} Let us introduce the following
process
\[\xi_t:=\int_0^te^{-\theta s}dG_s,\qquad t\geq0.\]
 Thus,  we can also write
\begin{eqnarray}\xi_t=e^{-\theta t}G_t+\theta Z_t,\qquad t\geq0.\label{xi and Z}\end{eqnarray}
\begin{remark}\label{remark1}Suppose that $G$ has H\"older continuous paths of order
 $\delta\in(\frac12,1]$. Then the process $X$  has $\delta$-H\"older continuous paths which implies
  that  the estimator $\widetilde{\theta}_t$ coincides with the LSE $\widehat{\theta}_t$
   by using (\ref{Integration by parts}). This property is satisfied in the cases of
    fractional Brownian motion with Hurst parameter $H>\frac12$, sub-fractional Brownian motion
    with  parameter $H>\frac12$ and bifractional Brownian motion with  parameters $(H,K)\in (0,1)^2$
     such that $HK>\frac12$ (see Section 3).\\
Indeed, let us prove that $X$  has $\delta$-H\"older continuous
paths. From (\ref{representation of X}), it suffices to prove that
the process $Z$ given in (\ref{expr Z}) has $\delta$-H\"older
continuous paths. Furthermore, Mean Value Theorem and the continuity
of $G$ entail that $Z$ has H\"older continuous paths of order $1$.
Thus, the result is obtained.
\end{remark}

The following theorem gives the strong consistency of the estimator
$\widetilde{\theta_t}$.

\begin{theorem}\label{consistency}
Assume that $(\mathcal{H}1)$ and $(\mathcal{H}2)$  hold and let
$\widetilde{\theta_t}$ be given by (\ref{estimator}). Then
\begin{eqnarray}\label{cv p.s.}\widetilde{\theta}_t\rightarrow \theta \mbox{ almost surely as }  t\to  \infty.
 \end{eqnarray}
\end{theorem}

In order to prove this theorem we make use of the following
technical lemmas. We will analyze separately the numerator and the
denominator in the right hand side of the estimator
(\ref{estimator}). The proofs of the following lemmas are given in
Appendix.

\begin{lemma}\label{convergence p.s of xi} Assume that $(\mathcal{H}1)$ and $(\mathcal{H}2)$  hold. Let $Z$ be the process
defined in (\ref{expr Z}). Then, $Z_{\infty}=\int_0^\infty
e^{-\theta s}G_sds$ is well-defined, and as $t\rightarrow\infty$
\begin{eqnarray} \label{cv Z} Z_t\longrightarrow Z_{\infty}\quad \mbox{almost surely and in $L^2(\Omega)$.}
\end{eqnarray}
Thus,  as $t\rightarrow\infty$
\begin{eqnarray}\label{cv xi}  \xi_t\longrightarrow \xi_\infty:=\theta Z_{\infty}
\quad \mbox{ almost surely and in $L^2(\Omega)$.}\end{eqnarray}
\end{lemma}
\begin{lemma}\label{convergence of dominating}
 Assume that $(\mathcal{H}1)$ and $(\mathcal{H}2)$  hold. Then, as $t\rightarrow \infty$, \[e^{-2\theta
t}\int_0^tX_s^2ds=e^{-2\theta t}\int_0^te^{2\theta
s}\xi_s^2ds\longrightarrow
\frac{\xi_{\infty}^2}{2\theta}=\frac{\theta}{2}Z_{\infty}^2\mbox{
almost surely. }\]
\end{lemma}

\begin{proof}[ Proof of Theorem \ref{consistency}]
By   (\ref{representation of X}), we can write
\begin{eqnarray}\widetilde{\theta}_t  =\frac{\left(G_t+\theta
e^{\theta t}Z_t\right)^2} {2\int_0^te^{2\theta s}\xi_s^2ds}
=\frac{\xi_t^2}{2e^{-2\theta t}\int_0^te^{2\theta
s}\xi_s^2ds}\label{rep thea tilde}.
\end{eqnarray}
The convergence (\ref{cv p.s.}) follows from  Lemmas
\ref{convergence p.s of xi} and  \ref{convergence of
dominating}.
\end{proof}

\subsection{Asymptotic distribution }
This section is devoted to the investigation of asymptotic
distribution of the estimator $\widetilde{\theta}_t$ of $\theta$.
Then, the following assumptions are required.
\begin{itemize}
\item[$(\mathcal{H}3)$] The limiting variance of $e^{-\theta t}\int_0^te^{\theta s}dG_s$  exists as
$t\rightarrow\infty$, i.e., there exists a constant $\sigma_G>0$ such
that
\[\lim_{t\rightarrow\infty}E\left[\left(e^{-\theta t}\int_0^te^{\theta s}dG_s\right)^2\right]\longrightarrow\sigma_G^2.\]
\item[$(\mathcal{H}4)$] For all fixed $s\geq0$
\begin{eqnarray*} \lim_{t\rightarrow\infty}E\left(G_se^{-\theta t}\int_0^te^{\theta
r}dG_r\right)=0.
\end{eqnarray*}
\end{itemize}

The next theorem proves that $\widetilde{\theta}_t$ is
asymptotically Cauchy.

\begin{theorem}\label{convergence distribution}Assume that $(\mathcal{H}1)$-$(\mathcal{H}4)$  hold. Then, as $t\rightarrow \infty$,
\begin{eqnarray}e^{\theta t}\left(\widetilde{\theta}_t-\theta\right)\overset{\texttt{law}}{\longrightarrow}\frac{2\sigma_G}{\sqrt{E\left(Z_\infty^2\right)}}\mathcal{C}(1),\label{cv in law}
\end{eqnarray}with  $\mathcal{C}(1)$ is the standard Cauchy distribution  with the probability density function
$\frac{1}{\pi (1+x^2)};\  x\in \mathbb{R}$.
\end{theorem}

In order to prove this theorem we need the following
lemmas. The proofs of these lemmas are given in Appendix.
\begin{lemma}\label{decomposition seconf moment of X}
Assume that $(\mathcal{H}1)$ holds. Then, for every $t\geq0$, we
have
\begin{eqnarray*}\frac12 X_t^2=\theta\int_0^t X_s^2 ds+\theta Z_t\int_0^te^{\theta
s}dG_s+R_t,
\end{eqnarray*} where $$R_t:=\frac12 G_t^2-\theta\int_0^t G_s^2 ds
+\theta^2\int_0^t ds\int_0^s dr G_sG_re^{-\theta(s-r)}.$$\end{lemma}

\begin{lemma}\label{lemma in law}Assume that $(\mathcal{H}1)$, $(\mathcal{H}3)$ and $(\mathcal{H}4)$  hold. Let $F$ be any $\sigma\{G\}$-measurable random variable such that $P(F<\infty)=1$. Then, as $t\rightarrow\infty$,  \begin{eqnarray*}\left(F, e^{-\theta t}\int_0^te^{\theta s}dG_s\right)\overset{\texttt{law}}{\longrightarrow}\left(F,\sigma_GN\right),
\end{eqnarray*}where $N\sim \mathcal{N}(0,1)$ is independent of $G$.
\end{lemma}

\begin{proof}[Proof of Theorem \ref{convergence distribution}] Using Lemma \ref{decomposition seconf moment of X}
  we can write
\begin{eqnarray*}
e^{\theta t}\left(\widetilde{\theta}_t-\theta\right)&=&\frac{
e^{-\theta t}\int_0^te^{\theta s}dG_s}{Z_\infty}\times\frac{\theta
Z_tZ_\infty}{e^{-2\theta t}\int_0^tX_s^2ds}+\frac{e^{-\theta
t}R_t}{e^{-2\theta
t}\int_0^tX_s^2ds}\\
&:=& a_t\times b_t+c_t.
\end{eqnarray*}
Lemma \ref{lemma in law} yields, as $t\rightarrow\infty$,
\[a_t\overset{\texttt{law}}{\longrightarrow}\sigma_G\frac{N}{Z_{\infty}},\]
where $N\sim \mathcal{N}(0,1)$ is independent of $G$, whereas Lemmas
\ref{convergence p.s of xi} and \ref{convergence of dominating}
imply that $b_t \longrightarrow 2$ almost surely as
$t\rightarrow\infty$. On the other hand , $e^{-\theta
t}R_t\longrightarrow0$ in $L^1(\Omega)$ as $t\rightarrow\infty$
because, as $t\rightarrow\infty$,
\begin{eqnarray*}
 &&e^{-\theta t} E\left(G_t^2\right)\leq ct^{2\gamma}e^{-\theta t} \longrightarrow0,\\
&&e^{-\theta t} \int_0^tE\left(G_s^2\right)ds\leq c\frac{t^{2\gamma+1}}{2\gamma+1}e^{-\theta t}\longrightarrow0,\\
\end{eqnarray*}
and
\begin{eqnarray*} e^{-\theta t} \int_0^t ds\int_0^s dr
E\left|G_sG_r\right|e^{-\theta(s-r)}\leq c e^{-\theta t} \int_0^t
ds\int_0^s dr (sr)^\gamma =\frac{c
t^{2\gamma+2}}{(\gamma+1)(2\gamma+2)}e^{-\theta t}\longrightarrow0.
\end{eqnarray*}
Combining this with Lemma \ref{convergence of dominating}, we obtain that $c_t \longrightarrow 0$ in probability as $t\rightarrow\infty$.\\
By plugging all these convergences together, we get that, as
$t\rightarrow\infty$,
\begin{eqnarray*}e^{\theta t}\left(\widetilde{\theta}_t-\theta\right)\overset{\texttt{law}}{\longrightarrow}2\sigma_G\frac{N}{Z_{\infty}}.
\end{eqnarray*}
Moreover, $Z_{\infty}\thicksim
\mathcal{N}\left(0,E\left[Z_{\infty}^2\right]\right)$, which
completes the proof.\end{proof}

\section{Applications to fractional Gaussian processes}
This section is devoted to some examples of the Gaussian process $G$
given in (\ref{OU}). We will need the following technical lemmas.
\begin{lemma}\label{lemma key1 of applications}Let $g : [0,\infty)\times[0,\infty)\longrightarrow\R$
 be a symmetric function such that $\frac{\partial  g }{\partial s }(s,r)$
 and $\frac{\partial^2 g }{\partial s\partial r}(s,r)$ exist on   $(0,\infty)\times[0,\infty)$. Then, for every $t\geq0$,
\begin{eqnarray}  \Delta_g(t)&:=&g(t,t)-2\theta e^{-\theta t}\int_0^t g(s,t)e^{\theta s}ds
+\theta^2 e^{-2\theta t}\int_0^t\int_0^t g(s,r)e^{\theta (s+r)}dr ds
\nonumber\\&=&2e^{-2\theta t}\int_0^te^{\theta s} \frac{\partial g
}{\partial s}(s,0)ds+2 e^{-2\theta t}\int_0^tdse^{\theta s}\int_0^s
dr\frac{\partial^2 g }{\partial s\partial r}(s,r)e^{\theta
r}.\label{key1}
\end{eqnarray}
\end{lemma}
\begin{proof} Set $h(s):=\int_0^s g(s,r)e^{\theta r}dr$. Combining  the integration by parts formula together with
\[\frac{\partial h }{\partial s}(s)=e^{\theta s}g(s,s)+\int_0^s \frac{\partial g }{\partial s}(s,r)e^{\theta r}dr,\] we obtain
\begin{eqnarray*} \Delta_g(t)&=&g(t,t)-2\theta e^{-2\theta t}\int_0^t  g(s,s) e^{2\theta s}ds
-2\theta e^{-2\theta t}\int_0^tdse^{\theta s}\int_0^s
dr\frac{\partial g }{\partial s}(s,r)e^{\theta r}\\&=&e^{-2\theta
t}\int_0^t \frac{\partial g(s,s)}{\partial s}(s)e^{2\theta s}ds
-2\theta e^{-2\theta t}\int_0^tdse^{\theta s}\int_0^s
dr\frac{\partial g }{\partial s}(s,r)e^{\theta r}.
\end{eqnarray*} Since $g$ is symmetric,
 $2\frac{\partial g }{\partial s}(s,r)1_{\{r=s\}}=\frac{\partial g(s,s)}{\partial s}(s)$ where
$\frac{\partial g(s,s)}{\partial s}$ is the derivative of the
function $s\longrightarrow g(s,s)$. Thus by using again the
integration by parts formula, the claim (\ref{key1}) is obtained.
\end{proof}
\begin{lemma}\label{lemma key2 of applications}
Let $\lambda>-1$. Define
\begin{eqnarray*}J_{\lambda}(t):=e^{-2\theta t}\int_0^t\int_0^te^{\theta s}e^{\theta r}|s-r|^{\lambda}drds; \quad
I_{\lambda}(t):=e^{-\theta t}\int_0^t e^{\theta
r}(t-r)^{\lambda}dr.
\end{eqnarray*}Then
\begin{eqnarray} \lim_{t\rightarrow\infty}J_{\lambda}(t)=\lim_{t\rightarrow\infty}\left(\frac{1}{\theta} I_{\lambda}(t)\right)=\frac{\Gamma(\lambda+1)}{\theta^{\lambda+2}} .\label{key2}
\end{eqnarray}
\end{lemma}
\begin{proof} Let $t\geq0$. We have
\begin{eqnarray*}J_{\lambda}(t)
&=&2e^{-2\theta t}\int_0^tdse^{\theta s}\int_0^s dr e^{\theta r}(s-r)^{\lambda}\\
&=&2e^{-2\theta t}\int_0^tdse^{2\theta s}\int_0^s dr e^{-\theta u}u^{\lambda}\\
&=&2e^{-2\theta t}\int_0^t du e^{-\theta u}u^{\lambda}\int_u^t dse^{2\theta s}\\
&=&\frac{1}{\theta}\left(\int_0^t u^{\lambda}e^{-\theta
u}du-e^{-2\theta t}\int_0^t u^{\lambda} e^{\theta u}du
\right)\\&\rightarrow&\frac{\Gamma(\lambda+1)}{\theta^{\lambda+2}}
\  \mbox{ as }t\rightarrow \infty,
\end{eqnarray*}
which proves (\ref{key2}).
\end{proof}
\subsection{Fractional Brownian motion}
The  fractional Brownian motion (fBm) $B^H=\left(B^H_t,
t\geq0\right)$ with Hurst parameter $H\in(0,1)$ is defined as a
centered Gaussian process starting from zero with covariance
\[E\left(B^H_tB^H_s\right)=\frac{1}{2}\left(t^{2H}+s^{2H}-|t-s|^{2H}\right).\]
Note that, when $H=\frac12$, $B^{\frac12}$ is a standard Brownian
motion.
\begin{proposition}\label{fBm case}Suppose that, in (\ref{OU}), the process $G$ is the fBm $B^H$. Then for all fixed  $H\in (0,1)$
the convergences (\ref{cv p.s.}) and (\ref{cv in law}) hold.
\end{proposition}
\begin{proof}
 By  Kolmogorov's continuity criterion
and the fact  $$E\left(B^H_t-B^H_s\right)^2=|s-t|^{2H};\ s,\
t\geq~0,$$ we deduce that $B^H$ has H\"older continuous paths of
order
$H-\varepsilon$ for all $\varepsilon\in(0,H)$. So, the process $B^H$ satisfies the assumptions
 $(\mathcal{H}1)$ and $(\mathcal{H}2)$. Thus   the strong consistence (\ref{cv p.s.})
 is obtained in the case when $G=B^H$.\\
For the convergence (\ref{cv in law}), it suffices to check
$(\mathcal{H}3)$ and $(\mathcal{H}4)$.
 Let us first compute the
limiting variance of $e^{-\theta t}\int_0^te^{\theta s}dB^H_s$  as
$t\rightarrow\infty$. We have
\begin{eqnarray}&&E\left[\left(e^{-\theta t}\int_0^te^{\theta
s}dB^H_s\right)^2\right]=
E\left[\left(e^{-\theta t}\left(e^{\theta
t}B^H_t-\theta\int_0^te^{\theta
s}B^H_sds\right)\right)^2\right]\nonumber\\&=& t^{2H}-2\theta
e^{-\theta t}\int_0^te^{\theta s}
E(B^H_sB^H_t)ds+\theta^2e^{-2\theta t}\int_0^t\int_0^te^{\theta
s}e^{\theta r}E(B^H_sB_r)dsdr\nonumber\\&=& t^{2H}-\theta e^{-\theta
t}\int_0^te^{\theta s}
\left(s^{2H}+t^{2H}-(t-s)^{2H}\right)ds\nonumber\\&&+\frac12\theta^2e^{-2\theta
t}\int_0^t\int_0^te^{\theta s}e^{\theta r}
\left(s^{2H}+r^{2H}-|r-s|^{2H}\right)dsdr\nonumber\\&=&
\Delta_{g_{B^H}}(t)+\theta I_{2H}(t)-\frac{\theta^2}{2}
J_{2H}(t),\label{Delta+I+J fBm}
\end{eqnarray}where $g_{B^H}(s,r)=\frac12 \left(s^{2H}+t^{2H}\right)$.\\
On the other hand,  (\ref{key1}) implies that
\begin{eqnarray}
\Delta_{g_{B^H}}(t)=2He^{-2\theta t}\int_0^ts^{2H-1}e^{\theta s}ds
&\rightarrow&0 \mbox{ as } t\rightarrow\infty.\label{Delta fBm}
\end{eqnarray}
Combining (\ref{key2})-(\ref{Delta fBm}),
we get for every $H\in(0,1)$
\begin{eqnarray*}E\left[\left(e^{-\theta t}\int_0^te^{\theta s}dB^H_s\right)^2\right]&\longrightarrow&\frac{H\Gamma(2H)}{\theta^{2H}}\quad\mbox{ as $t\rightarrow\infty$}.
\end{eqnarray*}
Hence, to finish the proof it remains to check that, for all fixed
$s\geq0$
\begin{eqnarray*} \label{cv Bs int}\lim_{t\rightarrow\infty}E\left(B^H_se^{-\theta t}\int_0^te^{\theta
r}dB^H_r\right)=0.
\end{eqnarray*}
Let us consider $s<t$.  Setting $f_{B^H}(s,r)=E(B^H_sB^H_r)$, it
follows from (\ref{Integration by parts}) that
\begin{eqnarray*} E\left(B^H_se^{-\theta t}\int_0^te^{\theta
r}dB^H_r\right) &=&f_{B^H}(s,t)- \theta e^{-\theta
t}\int_0^te^{\theta r}f_{B^H}(s,r)dr
\\&=&f_{B^H}(s,t)- \theta e^{-\theta
t}\int_s^te^{\theta r}f_{B^H}(s,r)dr- \theta e^{-\theta
t}\int_0^se^{\theta r}f_{B^H}(s,r)dr
\\&=&  e^{-\theta (t-s)}f_{B^H}(s,s)+  e^{-\theta t}\int_s^te^{\theta r}
\frac{\partial f_{B^H}}{\partial r}(s,r)dr- \theta e^{-\theta
t}\int_0^se^{\theta r}f_{B^H}(s,r)dr.
\end{eqnarray*}
It is clear that $e^{-\theta (t-s)}f_{B^H}(s,s)- \theta e^{-\theta t}\int_0^se^{\theta r}f_{B^H}(s,r)dr\longrightarrow0$ as $t\rightarrow\infty$.\\
Furthermore, if $H=\frac12$,   $\frac{\partial f_{B^H}}{\partial
r}(s,r)=0$ for every $r>s$. Then, for $H=\frac12$
\begin{eqnarray*}
e^{-\theta t}\int_s^te^{\theta r}\frac{\partial f_{B^H}}{\partial
r}(s,r)dr=0.
\end{eqnarray*}
Now, suppose that $H\in(0,\frac12)\cup(\frac12,1)$. Since
\begin{eqnarray*}
\int_s^te^{\theta r}\left|r^{2H-1}-(r-s)^{2H-1}\right|dr&\geq&
|2H-1|s\int_s^te^{\theta r}  r^{2H-2}  dr
\\&\geq& |2H-1|st^{2H-2}\int_s^te^{\theta r}dr\\&\rightarrow&\infty \mbox{ as } t\rightarrow\infty,
\end{eqnarray*}
we can apply L'H\^ospital's rule to obtain
\begin{eqnarray}
\lim_{t\rightarrow\infty}\left|e^{-\theta t}\int_s^te^{\theta
r}\frac{\partial f_{B^H}}{\partial r}(s,r)dr\right|
&=&\lim_{t\rightarrow\infty}\left|He^{-\theta t}\int_s^te^{\theta r}\left(r^{2H-1}-(r-s)^{2H-1}\right)dr\right|\nonumber\\
&\leq&\lim_{t\rightarrow\infty}\left(He^{-\theta t}\int_s^te^{\theta
r}\left| r^{2H-1}-(r-s)^{2H-1}\right|dr\right)
\nonumber\\&=&\lim_{t\rightarrow\infty}\left(\frac{H}{\theta}\left|
t^{2H-1}-(t-s)^{2H-1}\right|\right)
\nonumber\\&\leq&\lim_{t\rightarrow\infty}\left(\frac{sH|2H-1|}{\theta}
(t-s)^{2H-2}\right) \nonumber\\&\rightarrow&0 \mbox{ as }
t\rightarrow\infty \label{cv1 fBm case},
\end{eqnarray}
which finishes the proof of Proposition \ref{fBm case}.
\end{proof}
\subsection{Sub-fractional Brownian motion}
The sub-fractional Brownian motion (sfBm) $S^H$ with parameter
$H\in(0, 1)$   is a centred Gaussian process with covariance
function
\[E\left(S^H_tS^H_s\right)=t^{2H}+s^{2H}-\frac{1}{2}\left((t+s)^{2H}+|t-s|^{2H}\right).\]
Note that, when $H=\frac12$, $S^{\frac12}$ is a standard Brownian
motion.
\begin{proposition} Suppose that, in (\ref{OU}), the process $G$ is the sfBm $S^H$. Then for all fixed  $H\in (0,1)$
the convergences (\ref{cv p.s.}) and (\ref{cv in law}) hold.
\end{proposition}
\begin{proof}
 By  Kolmogorov's continuity criterion
and the fact  $$E\left(S^H_t-S^H_s\right)^2\leq
(2-2^{2H-1})|s-t|^{2H};\ s,\ t\geq~0,$$ we deduce that $S^H$ has
H\"older continuous paths of order
$H-\varepsilon$ for all $\varepsilon\in(0,H)$. So, the process $S^H$ satisfies the assumptions $(\mathcal{H}1)$ and $(\mathcal{H}2)$. Thus, by Theorem \ref{consistency} the convergence (\ref{cv p.s.}) is obtained.\\
To prove (\ref{cv in law}), it suffices to check $(\mathcal{H}3)$
and $(\mathcal{H}4)$.  The case $H=\frac12$  has already been
established above. Suppose now that
$H\in(0,\frac12)\cup(\frac12,1)$. Using the same argument as in
(\ref{Delta+I+J fBm}), we get
\begin{eqnarray}E\left[\left(e^{-\theta t}\int_0^te^{\theta s}dS^H_s\right)^2\right]&=&
\Delta_{g_{S^H}}(t)+\theta I_{2H}(t)-\frac{\theta^2}{2}
J_{2H}(t),\label{Delta+I+J sfBm}
\end{eqnarray}where $g_{S^H}(s,r)= s^{2H}+t^{2H}-\frac12(s+t)^{2H} $.\\
Moreover, we have
\begin{eqnarray*}
\Delta_{g_{S^H}}(t)=2He^{-2\theta t}\int_0^ts^{2H-1}e^{\theta
s}ds-2H(2H-1)e^{-2\theta t}\int_0^tdse^{\theta s}\int_0^sdre^{\theta
r}(s+r)^{2H-2}.
\end{eqnarray*}
It is easy to see that $2He^{-2\theta t}\int_0^ts^{2H-1}e^{\theta s}ds \rightarrow0 \mbox{ as } t\rightarrow\infty$.\\
Furthermore, using the fact that
\begin{eqnarray*}
 \int_0^tdse^{\theta s}\int_0^sdre^{\theta r}(s+r)^{2H-2}
&\geq&(2t)^{2H-2}\int_0^tdse^{\theta s}\int_0^sdre^{\theta r}
\\&=&\frac{(2t)^{2H-2}}{2}\left(\int_0^te^{2\theta s}ds\right)^2
\\&\rightarrow&\infty \mbox{ as } t\rightarrow\infty,
\end{eqnarray*}
 L'H\^ospital's rule entails
\begin{eqnarray*}
\lim_{t\rightarrow\infty}\left(e^{-2\theta t}\int_0^tdse^{\theta
s}\int_0^sdre^{\theta r}(s+r)^{2H-2}\right)
&=&\lim_{t\rightarrow\infty}\left(\frac{1}{2\theta}e^{-\theta t}\int_0^te^{\theta r}(t+r)^{2H-2}dr\right)\\
&\leq&\lim_{t\rightarrow\infty}\left(\frac{t^{2H-2}}{2\theta}e^{-\theta
t}\int_0^te^{\theta r}dr\right)
\\&\rightarrow0& \mbox{ as } t\rightarrow\infty.
\end{eqnarray*}
Thus,  we deduce that
\begin{eqnarray}
\Delta_{g_{S^H}}(t) &\rightarrow&0 \mbox{ as }
t\rightarrow\infty.\label{Delta sfBm}
\end{eqnarray}
Combining (\ref{Delta+I+J sfBm}), (\ref{Delta sfBm}) and
(\ref{key2}) we get
\begin{eqnarray*}E\left[\left(e^{-\theta t}\int_0^te^{\theta s}dS^H_s\right)^2\right]&\longrightarrow&\frac{H\Gamma(2H)}{\theta^{2H}}\quad\mbox{ as $t\rightarrow\infty$}.
\end{eqnarray*}
Hence, to finish the proof it remains to check that, for all fixed
$s\geq0$
\begin{eqnarray*} \label{cv Bs int}\lim_{t\rightarrow\infty}E\left(S^H_se^{-\theta t}\int_0^te^{\theta
r}dS^H_r\right)=0.
\end{eqnarray*}
Let us consider $s<t$ and let $f_{S^H}(s,r)=E(S^H_sS^H_r)$. Then, as
in the fBm case, we can write
\begin{eqnarray*} E\left(S^H_se^{-\theta t}\int_0^te^{\theta
r}dS^H_r\right) &=& e^{-\theta (t-s)}f(s,s)+  e^{-\theta
t}\int_s^te^{\theta r}\frac{\partial f_{S^H}}{\partial r}(s,r)dr-
\theta e^{-\theta t}\int_0^se^{\theta r}f_{S^H}(s,r)dr.
\end{eqnarray*}
It is clear   that $e^{-\theta (t-s)}f_{S^H}(s,s)- \theta e^{-\theta t}\int_0^se^{\theta r}f_{S^H}(s,r)dr\longrightarrow0$ as $t\rightarrow\infty$.\\
On the other hand, since
\begin{eqnarray*}
e^{-\theta t}\int_s^te^{\theta r}\frac{\partial f_{S^H}}{\partial
r}(s,r)dr=\frac{H}{2}e^{-\theta t}\int_s^te^{\theta
r}\left(2r^{2H-1}-(r+s)^{2H-1}-(r-s)^{2H-1}\right)dr,
\end{eqnarray*}
the same argument as in (\ref{cv1 fBm case})  leads to
\begin{eqnarray*}
e^{-\theta t}\int_s^te^{\theta r}\frac{\partial f_{S^H}}{\partial
r}(s,r)dr &\longrightarrow&0\quad \mbox{ as } t \rightarrow\infty,
\end{eqnarray*} which finishes the proof.
\end{proof}
\subsection{Bifractional Browian motion}
Let $B^{H,K} = \left( B^{H,K}_t, t\geq 0 \right)$ be a  bifractional
Brownian motion (bifBm)  with parameters $H\in (0, 1)$ and
$K\in(0,1]$. This means that $B^{H,K}$ is a centered Gaussian
process with the covariance function
\begin{eqnarray*}
E(B^{H,K}_sB^{H,K}_t)=\frac{1}{2^K}\left(\left(t^{2H}+s^{2H}\right)^K-|t-s|^{2HK}\right).
\end{eqnarray*}
  The case $K = 1$ corresponds to the  fBm  with Hurst parameter $H$. The process $B^{H,K}$ verifies \begin{eqnarray*}
 E\left(\left|B^{H,K}_t-B^{H,K}_s\right|^2\right)\leq
2^{1-K}|t-s|^{2HK},
\end{eqnarray*}
so $B^{H,K}$ has $(HK-\varepsilon)-$H\"{o}lder continuous paths for
any $\varepsilon\in (0,HK)$ thanks to Kolmogorov's continuity
criterion. The bifBm $B^{H,K}$ can be extended for $1 < K < 2$ with
$H\in (0, 1)$ and $HK\in (0, 1)$ (see \cite{BE} and \cite{LV}).
\begin{proposition} Suppose that, in (\ref{OU}), the process $G$ is the bifBm $B^{H,K}$. Then
the convergences (\ref{cv p.s.}) and (\ref{cv in law}) hold true for
all fixed  $(H,K)\in (0,1)^2$.
\end{proposition}
\begin{proof}
Since $B^{H,K}$ has H\"older continuous paths of order
$HK-\varepsilon$ for all $\varepsilon\in(0,HK)$, it satisfies the assumptions $(\mathcal{H}1)$ and $(\mathcal{H}2)$.
Thus the convergence (\ref{cv p.s.}) is satisfied.\\
To prove (\ref{cv in law}), it suffices to check $(\mathcal{H}3)$
and $(\mathcal{H}4)$. Using the same argument as in (\ref{Delta+I+J
fBm}), we have
\begin{eqnarray}E\left[\left(e^{-\theta t}\int_0^te^{\theta s}dB^{H,K}_s\right)^2\right]&=&
\Delta_{g_{B^{H,K}}}(t)+2^{1-K}\theta I_{2HK}(t)-2^{-K}\theta^2
J_{2HK}(t),\label{Delta+I+J bifBm}
\end{eqnarray}where $g_{B^{H,K}}(s,r)= \frac{1}{2^K}\left(s^{2H}+r^{2H}\right)^K $.\\
On the other hand,
\begin{eqnarray*}
\Delta_{g_{B^{H,K}}}(t)&=&2^{2-K}HKe^{-2\theta
t}\int_0^ts^{2HK-1}e^{\theta s}ds
\\&&-2^{3-K}H^2K(K-1)e^{-2\theta t}\int_0^tdse^{\theta s}\int_0^sdre^{\theta r}(sr)^{2H-1}
\left(s^{2H}+r^{2H}\right)^{K-2}.
\end{eqnarray*}
The convergence  $2^{2-K}HKe^{-2\theta t}\int_0^ts^{2HK-1}e^{\theta s}ds \rightarrow0 \mbox{ as } t\rightarrow\infty$ is immediate.\\
Also,  it is straightforward to check that there exists a constant
$C_{H,K}$ depending on $H,K$ such that
\begin{eqnarray*}
\int_0^tdse^{\theta s}\int_0^sdre^{\theta
r}(sr)^{2H-1}\left(s^{2H}+r^{2H}\right)^{K-2}
&\geq&C_{H,K}t^{2HK-2}\int_{\frac{t}{2}}^tdse^{\theta
s}\int_{\frac{s}{2}}^sdre^{\theta r}
\\&\geq&\frac{C_{H,K}}{2}t^{2HK-2}\int_{\frac{t}{2}}^t se^{\frac{3\theta}{2} s}ds
\\&\rightarrow\infty& \mbox{ as } t\rightarrow\infty.
\end{eqnarray*}
So, we can apply  L'H\^ospital's rule  to obtain
\begin{eqnarray*}
&&\lim_{t\rightarrow\infty}\left(e^{-2\theta t}\int_0^tdse^{\theta
s}\int_0^sdre^{\theta
r}(sr)^{2H-1}\left(s^{2H}+r^{2H}\right)^{K-2}\right)
\\&=&\lim_{t\rightarrow\infty}\left(\frac{e^{-\theta t}}{2\theta}\int_0^te^{\theta r}(tr)^{2H-1}\left(t^{2H}+r^{2H}\right)^{K-2}dr\right)\\
&\leq&\lim_{t\rightarrow\infty}\left(\frac{2^{K-3}e^{-\theta
t}}{\theta}\int_0^te^{\theta r}(tr)^{HK-1}dr\right)
\\&\leq&\lim_{t\rightarrow\infty}\left(\frac{2^{K-3}t^{2HK-2}}{\theta}e^{-\theta t}\int_0^te^{\theta r}dr\right)
\\&\rightarrow0& \mbox{ as } t\rightarrow\infty.
\end{eqnarray*}
Hence, for every $(H,K)\in (0,1)^2$
\begin{eqnarray}
\Delta_{g_{B^{H,K}}}(t) &\rightarrow&0 \mbox{ as }
t\rightarrow\infty.\label{Delta bifBm}
\end{eqnarray}
Consequently, (\ref{Delta+I+J bifBm}), (\ref{Delta bifBm}) and
(\ref{key2}) imply
\begin{eqnarray*}E\left[\left(e^{-\theta t}\int_0^te^{\theta s}dB^{H,K}_s\right)^2\right]&\longrightarrow&\frac{HK\Gamma(2HK)}{\theta^{2HK}}\quad\mbox{ as $t\rightarrow\infty$}.
\end{eqnarray*}
Hence, to finish the proof it remains to check that, for all fixed
$s\geq0$
\begin{eqnarray*} \label{cv Bs int}\lim_{t\rightarrow\infty}E\left(B^{H,K}_se^{-\theta t}\int_0^te^{\theta
r}dB^{H,K}_r\right)=0.
\end{eqnarray*}
Let us consider $s<t$ and let
$f_{B^{H,K}}(s,r)=E(B^{H,K}_sB^{H,K}_r)$. Then, as in the fBm case,
we can write
\begin{eqnarray*} &&E\left(B^{H,K}_se^{-\theta t}\int_0^te^{\theta
r}dB^{H,K}_r\right) \\&=& e^{-\theta (t-s)}f_{B^{H,K}}(s,s)+
e^{-\theta t}\int_s^te^{\theta r}\frac{\partial
f_{B^{H,K}}}{\partial r}(s,r)dr- \theta e^{-\theta
t}\int_0^se^{\theta r}f_{B^{H,K}}(s,r)dr.
\end{eqnarray*}
We have $e^{-\theta (t-s)}f_{B^{H,K}}(s,s)- \theta e^{-\theta t}\int_0^se^{\theta r}f_{B^{H,K}}(s,r)dr\longrightarrow0$ as $t\rightarrow\infty$.\\
Also,
\begin{eqnarray*}
e^{-\theta t}\int_s^te^{\theta r}\frac{\partial
f_{B^{H,K}}}{\partial r}(s,r)dr=2^{1-K}HKe^{-\theta
t}\int_s^te^{\theta
r}\left(r^{2H-1}\left(s^{2H}+r^{2H}\right)^{K-1}-(r-s)^{2HK-1}\right)dr.
\end{eqnarray*}
Hence, if $HK<\frac12$, L'H\^ospital's rule leads to
\begin{eqnarray*}
\left|e^{-\theta t}\int_s^te^{\theta r}\frac{\partial
f_{B^{H,K}}}{\partial r}(s,r)dr\right|
&\leq&2^{1-K}HKe^{-\theta t}\int_s^te^{\theta r}\left(r^{2HK-1}+(r-s)^{2HK-1}\right)dr\\
&\leq&2^{2-K}HKe^{-\theta t}\int_s^te^{\theta r}(r-s)^{2HK-1}dr\\
&\longrightarrow&\lim_{t\rightarrow\infty}\left(\frac{2^{2-K}HK}{\theta}(t-s)^{2HK-1}\right)=0\quad
\mbox{ as } t \rightarrow\infty.
\end{eqnarray*}
If $HK=\frac12$,
 \begin{eqnarray*}
\left|e^{-\theta t}\int_s^te^{\theta r}\frac{\partial
f_{B^{H,K}}}{\partial r}(s,r)dr\right|
&=&2^{-K}e^{-\theta t} \int_s^te^{\theta r}\left(1-\left(1+\left(\frac{s}{r}\right)^{2H}\right)^{K-1}\right)dr \\
&\longrightarrow&0\quad \mbox{ as } t \rightarrow\infty.
\end{eqnarray*}
 The last convergence comes from  the fact that for $r$ large, $$1-\left(1+\left(\frac{s}{r}\right)^{2H}\right)^{K-1}\leq1-\left(1+\frac{s}{r}\right)^{K-1}\sim (1-K)\frac{s}{r},$$ and L'H\^ospital's rule.
 Similarly, if $HK>\frac12$,
\begin{eqnarray*}
&&\left|e^{-\theta t}\int_s^te^{\theta r}\frac{\partial f_{B^{H,K}}}{\partial r}(s,r)dr\right|\\
&\leq&2^{1-K}HKe^{-\theta t}\int_s^te^{\theta r}r^{2HK-1}\left|\left(1+\left(\frac{s}{r}\right)^{2H}\right)^{K-1}-\left(1-\frac{s}{r}\right)^{2HK-1}\right|dr\\
&\longrightarrow&0\quad \mbox{ as } t \rightarrow\infty,
\end{eqnarray*}
which completes the proof.
\end{proof}

\noindent {\bf Acknowledgement : }The authors would like to thank
the editor and  two anonymous referees   for their careful reading
of the manuscript and for their valuable suggestions and remarks.
The second author  would also like to acknowledge the financial
support of the  University of  Rouen. The third author was supported
by Hassan II Academy of Science and Technology.

 \renewcommand{\theequation}{A-\arabic{equation}}
  \setcounter{equation}{0}
  \section*{Appendix}

In this section we present some calculations used in the paper.\\

For any $\alpha\in (0,1]$, we denote by $\mathscr{H}^\alpha([0,T])$
the set of $\alpha$-H\"older continuous functions, that is, the set
of functions $f:[0,T]\to\R$ such that
\[
|f|_\alpha := \sup_{0\leq s<t\leq
T}\frac{|f(t)-f(s)|}{(t-s)^{\alpha}}<\infty.
\]
We also set $|f|_\infty=\sup_{t\in[0,T]}|f(t)|$, and we equip
$\mathscr{H}^\alpha([0,T])$ with the norm $\|f\|_\alpha :=
|f|_\alpha + |f|_\infty.$\\
 Let $f\in\mathscr{H}^\alpha([0,T])$, and
consider the operator $T_f:\mathcal{C}^1([0,T])
\to\mathcal{C}^0([0,T])$ defined as
\[
T_f(g)(t)=\int_0^t f(u)g'(u)du, \quad t\in[0,T].
\]
It can be shown (see, e.g., \cite[Section 3.1]{nourdin}) that, for
any $\beta\in(1-\alpha,1)$, there exists a constant
$C_{\alpha,\beta,T}>0$ depending only on $\alpha$, $\beta$ and $T$
such that, for any $g\in\mathscr{H}^\beta([0,T])$,
\[
\left\|\int_0^\cdot f(u)g'(u)du\right\|_\beta \leq
C_{\alpha,\beta,T} \|f\|_\alpha \|g\|_\beta.
\]
We deduce that, for any $\alpha\in (0,1)$, any
$f\in\mathscr{H}^\alpha([0,T])$ and any $\beta\in(1-\alpha,1)$, the
linear operator
$T_f:\mathcal{C}^1([0,T])\subset\mathscr{H}^\beta([0,T])\to
\mathscr{H}^\beta([0,T])$, defined as $T_f(g)=\int_0^\cdot
f(u)g'(u)du$, is continuous with respect to the norm
$\|\cdot\|_\beta$. By density, it extends (in an unique way) to an
operator defined on $\mathscr{H}^\beta$. As consequence, if
$f\in\mathscr{H}^\alpha([0,T])$, if $g\in\mathscr{H}^\beta([0,T])$
and if $\alpha+\beta>1$, then the (so-called) Young integral
$\int_0^\cdot f(u)dg(u)$ is well-defined as being $T_f(g)$.

The Young integral obeys the following formula. Let
$f\in\mathscr{H}^\alpha([0,T])$ with $\alpha\in(0,1)$ and
$g\in\mathscr{H}^\beta([0,T])$ for all $\beta\in(0,1)$. Then
$\int_0^. g_udf_u$ and $\int_0^. f_u dg_u$ are well-defined as the Young
integrals. Moreover, for all $t\in[0,T]$,
\begin{eqnarray}\label{Integration by parts}
f_tg_t=f_0g_0+\int_0^t g_udf_u+\int_0^t f_u dg_u.
\end{eqnarray}

\begin{proof}[Proof of Lemma \ref{convergence p.s of xi}] We first notice that the integral $Z_{\infty}=\int_0^{\infty}e^{-\theta s}G_sds$
is well-defined because
$$\int_0^\infty e^{-\theta s}E(|G_s|)ds\leq \sqrt{c}\int_0^\infty s^{\gamma} e^{-\theta s}ds<\infty.$$
Now, we prove (\ref{cv Z}). By using Borel-Cantelli's lemma, it is
sufficient to prove that, for any $\varepsilon > 0$,
\begin{eqnarray*}\label{BC}
\sum_{n\geq0} P\left(\sup_{n\leq t \leq
n+1}\left|\int_{t}^{\infty}e^{-\theta
s}G_sds\right|>\varepsilon\right) <\infty.
\end{eqnarray*}
Notice that for every $\varepsilon > 0$,
\begin{eqnarray*}
E\left(\sup_{n\leq t \leq {n+1}}\left|\int_{t}^{\infty}e^{-\theta
s}G_sds\right|\right)&\leq& E\left(\int_{n}^{\infty}e^{-\theta
s}\left|G_s\right|ds\right)\\
&\leq&\sqrt{c} \int_{n}^{\infty}e^{-\theta
s}s^{\gamma}ds\\
&\leq&\sqrt{c}e^{-\frac{\theta}{2}n}\int_{0}^{\infty}e^{-\frac{\theta}{2}
s}s^{\gamma}ds\\&=&\sqrt{c}\Gamma\left(1+\gamma\right)\left(\frac{2}{\theta}\right)^{1+\gamma}e^{-\frac{\theta}{2}n}.
\end{eqnarray*}
Consequently,
\begin{eqnarray*}
\sum_{n\geq0} P\left(\sup_{n\leq t \leq
{n+1}}\left|\int_{t}^{\infty}e^{-\theta
s}dG_s\right|>\varepsilon\right)
&\leq& \varepsilon^{-1}\sum_{n\geq0}E\left(\sup_{n\leq t \leq {n+1}}\left|\int_{t}^{\infty}e^{-\theta s}dG_s\right|\right)\\
&\leq&
\varepsilon^{-1}\sqrt{c}\Gamma\left(1+\gamma\right)\left(\frac{2}{\theta}\right)^{1+\gamma}\sum_{n\geq0}
e^{-\frac{\theta}{2}n} <\infty,
\end{eqnarray*}
which implies that $Z_t\longrightarrow Z_{\infty}$ almost surely as
$t\rightarrow\infty$. Moreover, since
\begin{eqnarray*}E\left[\left(Z_t-Z_{\infty}\right)^2\right]&=&  \int_t^\infty\int_t^\infty e^{-\theta r}e^{-\theta s}E\left(G_rG_s\right)drds
 \\&\leq& c \int_t^\infty\int_t^\infty e^{-\theta r}e^{-\theta s} \left(rs\right)^\gamma drds
 \\&=&c \left(\int_t^\infty  e^{-\theta s} s^\gamma  ds\right)^2
 \\&\rightarrow&0 \mbox{ as } t\rightarrow\infty,
 \end{eqnarray*} the proof of the claim (\ref{cv Z}) is finished.
The convergence (\ref{cv xi}) is a direct consequence of
(\ref{cv Z}) and (\ref{xi and Z}). Thus the proof of Lemma
\ref{convergence p.s of xi} is done.
\end{proof}

\noindent \begin{proof}[Proof of Lemma \ref{convergence of
dominating}] It follows from  (\ref{cv xi}) that
$\xi_{\infty}\backsim
\mathcal{N}\left(0,E\left[\xi_{\infty}^2\right]\right)$, where
 \begin{eqnarray*}E\left[\xi_{\infty}^2\right]=\theta^2E\left[Z_{\infty}^2\right]&=& \theta^2 \int_0^\infty\int_0^\infty e^{-\theta r}e^{-\theta s}E\left(G_rG_s\right)drds
 \\&\leq& c\theta^2 \int_0^\infty\int_0^\infty e^{-\theta r}e^{-\theta s} \left(rs\right)^\gamma drds
 \\&=&c\left(\frac{\Gamma(\gamma+1)}{\theta^{\gamma}}\right)^2<\infty.
 \end{eqnarray*}This implies that
\begin{eqnarray}\label{xi_infty<fini}P(\xi_{\infty}=0)=0.
\end{eqnarray}
The continuity of $\xi$ entails that, for every $t\geq0$
\begin{eqnarray}\label{minoration xi}\int_0^te^{2\theta
s}\xi_s^2ds\geq
 \int_{\frac{t}{2}}^te^{2\theta s}\xi_s^2ds\geq\frac{t}{2}e^{\theta t}\left(\inf_{\frac{t}{2}\leq s\leq t}\xi_{s}^2\right)\ \mbox{ almost surely}.\end{eqnarray} Furthermore, the continuity of $\xi$ and (\ref{cv xi}) yield  \[\lim_{t\rightarrow\infty}\left(\inf_{\frac{t}{2}\leq s\leq t}\xi_{s}^2\right)=\xi_{\infty}^2\ \mbox{ almost surely}.\]Combining  this last convergence with (\ref{xi_infty<fini}) and (\ref{minoration xi}), we deduce that
  \[\lim_{t\rightarrow\infty}\int_0^te^{2\theta s}\xi_s^2ds=\infty\ \mbox{ almost surely}.\]
Hence, we can use L'H\^ospital's rule to obtain
\begin{eqnarray*}\lim_{t\rightarrow\infty}\frac{\int_0^te^{2\theta s}\xi_s^2ds}{e^{2\theta t}} = \lim_{t\rightarrow\infty}\frac{\xi_t^2}{2\theta}
 = \frac{\xi_{\infty}^2}{2\theta}\quad \mbox{ almost surely,}
\end{eqnarray*}which completes the proof of Lemma \ref{convergence of dominating}.\end{proof}

\begin{proof}[Proof of Lemma \ref{decomposition seconf moment of X}]
Let $t\geq0$. Setting $\eta_t=\int_0^t X_s ds$,   the equation
(\ref{OU}) leads to
\begin{eqnarray*}\frac12 X_t^2=\frac12 \theta^2\eta_t^2+\frac12G_t^2+\theta \eta_tG_t.\end{eqnarray*}
Moreover, (\ref{Integration by parts}) and (\ref{OU}) entail
\begin{eqnarray*}\frac12 \eta_t^2=\int_0^t \eta_sd\eta_s=\int_0^t \eta_s
X_sds=\theta^{-1}\left(\int_0^t X_s^2ds-\int_0^t G_s X_s
ds\right).
\end{eqnarray*}Define   $Y_t:=\int_0^t e^{\theta s}G_sds$. Then,  by (\ref{representation of X}) and (\ref{Integration by parts})
\begin{eqnarray*}\int_0^t G_s X_s ds&=&\int_0^t  G_s\left( G_s+\theta e^{\theta s}Z_s\right) ds
\\&=& \int_0^t G_s^2ds+\theta \int_0^t e^{\theta s} G_sZ_sds
\\&=& \int_0^t G_s^2ds+\theta \int_0^t    Z_sdY_s
\\&=& \int_0^t G_s^2ds+\theta Z_tY_t-\theta \int_0^t   Y_sdZ_s
\\&=& \int_0^t G_s^2ds+\theta Z_tY_t-\theta \int_0^t ds\int_0^s dr G_sG_re^{-\theta(s-r)}.
\end{eqnarray*}
Thus, we deduce that
\begin{eqnarray}\label{equ on R}\frac12 X_t^2=\theta \int_0^t  X_s^2 ds-\theta^2Z_tY_t +\theta
\eta_tG_t+ R_t.
\end{eqnarray}
On the other hand, by (\ref{OU}) and (\ref{representation of X}) we
get
\begin{eqnarray*}\theta\eta_tG_t=  G_t\left(X_t-G_t\right)
=-\theta e^{\theta t} G_tZ_t.
\end{eqnarray*}
This implies that
\begin{eqnarray*}-\theta^2Z_tY_t +\theta
\eta_tG_t=-\theta Z_t(\theta Y_t-e^{\theta t}G_t)&=&   \theta
Z_t\int_0^te^{\theta s}dG_s.
\end{eqnarray*}
Combining this with (\ref{equ on R})  the proof of Lemma
\ref{decomposition seconf moment of X} is done.
\end{proof}

\begin{proof}[Proof of Lemma \ref{lemma in law}]For any $d\geq1$, $s_1\ldots
s_d\in[0,\infty)$, we shall prove that, as $t\rightarrow\infty$,
\begin{eqnarray}\label{law technic}\left(B_{s_1},\ldots,B_{s_d},e^{-\theta t}\int_0^te^{\theta s}dG_s\right)\overset{\texttt{law}}{\longrightarrow}
\left(B_{s_1},\ldots,B_{s_d},\sigma_GN\right),
\end{eqnarray} which is enough to lead to the desired conclusion. Because the left-hand side in the previous
convergence is a Gaussian vector (see proof of Lemma 7 in
\cite{bridge}), to get (\ref{law technic}) it is sufficient to check
the convergence of its covariance matrix. Thus, the assumptions
$(\mathcal{H}3)$ and $(\mathcal{H}4)$ complete the proof.
\end{proof}

\bibliographystyle{amsplain}

\end{document}